\documentclass[a4paper, 12pt]{amsart}
\usepackage[utf8]{inputenc}
\usepackage{Settings}

\title{$G$-gerbes on perfectoid spaces}
\author{Xiaohuan Long}
\address{(X. Long) Academy of Mathematics and Systems Science, Chinese Academy of Sciences, Beijing 100190, China.}
\email{longxiaohuan@amss.ac.cn}
\author{Yibin Wang}
\address{(Y. Wang) School of Mathematical Sciences, Peking University, Beijing, China.}
\email{2200010606@stu.pku.edu.cn}
\author{Xiangdong Wu}
\address{(X. Wu) Academy of Mathematics and Systems Science, Chinese Academy of Sciences, Beijing 100190, China.}
\email{wuxiangdong@amss.ac.cn}
\author{Ru Yi}
\address{(R. Yi) School of Mathematical Sciences, Fudan University, Shanghai, China.}
\email{22300180148@m.fudan.edu.cn}
\date{\today}
\begin{document}
\begin{abstract}
    Let $K$ be a complete non-archimedean field over $\mathbb{Q}_p$, $G$ be a rigid group over $K$, and $X$ be a perfectoid space over $K$. We consider the natural morphism of sites $\nu\colon X_v \to X_\et$. It is known from work of Heuer that the direct image functor $\nu_*$ induces an equivalence of the categories of $G$-torsors. In this article, we show that there is an equivalence of 2-categories of $G$-gerbes on these two topologies.
\end{abstract}
\maketitle


\section{Introduction}
We fix a non-archimedean field $K/\bbQ_p$. Let $X/K$ be a perfectoid space, $G/K$ a rigid group. Then we have a natural morphism of sites $\nu\colon X_v\to X_{\et}$. Consider the diamond associated to $G$
\[G^\diamondsuit\colon Y\mapsto \Hom_K(Y,G)\]
defined on $\Perfd_K$, the category of perfectoid spaces over $K$. Let $G_v$ and $G_\et$ denote the restrictions of $G^\dia$ on $X_v$ and $X_\et$. By definition, $G_\et=\nu_*G_v$. 

Heuer in \cite{Heu22} shows that each $G_v$-torsor is \'etale locally trivial. Hence the functor 
\[
    \nu_*\colon \Tors(X_v;G_v) \to \Tors(X_\et;G_\et)
\]
is well-defined, and it admits an obvious quasi-inverse as a composite of $\nu^*$ followed by the contracted product:
\[
    \Tors(X_\et;G_\et) \xrightarrow{\nu^*} \Tors(X_v; \nu^*G_\et) \xrightarrow{- \wedge^{\nu^*G_\et} G_v} \Tors(X_v; G_v).
\]
In other words, $\nu_*$ induces an equivalence of \'etale stacks
\begin{equation}\label{eq: BG}
    \nu_*\bbB G_v\xrightarrow{\sim} \bbB G_\et.
\end{equation}

In this article, we establish an analogous equivalence for $G$-gerbes, where a $G$-gerbe is a stack of groupoid that is locally isomorphic to $\mathbb{B}G$ with trivial band. More precisely, we obtain the following.

\begin{theo}\label{thm: MainTheoremGerbe}
    The functor $\nu_*$ induces an equivalence of 2-categories
    \[
        \nu_*\colon \Gerb(X_v; G_v) \to \Gerb(X_\et; G_\et).
    \]
\end{theo}
\begin{rema}
    The perfectoidness of $X$ is necessary. See Remark \ref{rema: perfectoidAssumtionNecessary}.
\end{rema}

We restrict our focus on $1$-gerbes in the sense of Giraud, since for $n \ge 2$, $n$-gerbes on a site $E$ are banded by abelian groups, and for any abelian group $A$ on $E$, $n$-$A$-gerbes on $E$ are classified by abelian cohomology groups $H^{n+1}(E,A)$. (See \cite[Corollary 7.2.2.27]{HTT}).

Recall that Heuer in \cite{Heu22} shows that for any adic space over $K$, the categories of $G$-torsors on $X_v$ and $X_{\mathrm{qpro\et}}$ are equivalent, and if moreover $X$ is locally Noetherian, then the categories of $G$-torsors on $X_v$ and $X_{\mathrm{pro\et}}$ are equivalent. We also obtain the following analogue results:
\begin{coro}
    Let $X$ be any adic space over $K$, then the categories of $G$-gerbes on $X_v$ and $X_{\mathrm{qpro\et}}$ (and $X_{\mathrm{pro\et}}$ if $X$ is locally Noetherian) are equivalent.
\end{coro}
\subsection{Reduction via the centre}
These results are crucially depended on the centres of $G_v$ and $G_\et$.

\begin{theo}\label{thm: Centres}
    Let $G$ be a rigid group over $K$. Then
    \begin{enumerate}[(1)]
        \item there is a Zariski-closed rigid subgroup $Z_\rig \subset G$ such that $Z_\rig^\dia$ is the centre of $G^\dia$;
        \item the restrictions of $Z_\rig^\dia$ on $X_v$ and $X_\et$ are the centres of $G_v$ and $G_\et$ respectively.
    \end{enumerate}
\end{theo}
Let us briefly deduce Theorem \ref{thm: MainTheoremGerbe} from Theorem \ref{thm: Centres}.

By a result of Giraud \cite[proposition 3.1.10]{Gir71}, every $G_v$-gerbe on $X_v$ is \'etale locally trivial. Combining with Heuer's main result, we find that $\nu_*$ is well-defined. The fully faithfulness is a consequence of Theorem \ref{thm: Centres}. Note that by \eqref{eq: BG}, $\nu_*$ is \emph{\'etale locally essentially surjective}, namely, any $G_\et$-gerbe on $X_\et$ is \'etale locally of the form $\nu_*\bbB G_v$. We deduce that $\nu_*$ is essentially surjective by $2$-descent.

\begin{rema}
    We \emph{can not} define a quasi-inverse of $\nu_*$ in a similar way as for torsors, since non-abelian $H^2$ lacks functoriality.
\end{rema}

\subsection{Setup and Notation}
Throughout let $p$ be a prime and $K$ a complete non-archimedean field over $\mathbb{Q}_p$. We put $\Spa K \coloneqq \Spa(K,K^{\circ})$. For any finite extension $L$ of $K$, we simply write $L$ for the Huber pair $(L,L^{\circ})$, and write $\Spa L$ for $\Spa(L,L^\circ)$ if there is no risk of confusion. By an adic space over $K$ we mean an adic space over $\Spa(K,K^\circ)$. By a rigid space over $K$ we mean an adic space locally of topologically finite type over $\Spa(K,K^\circ)$.

\subsection*{Acknowledgements.} 
This work is one of the projects in the “Algebra and Number Theory" summer school jointly organized by Chinese Academy of Sciences and Peking University. We want to express our gratitude to Professor Shou-wu Zhang and Professor Weizhe Zheng. Without their efforts, the summer school would not have been such a success. We are also grateful to our supervisor, Professor Ning Guo, for suggesting this question and continuous guidance. We thank Professor Ben Heuer, Professor Xu Shen and Professor Weizhe Zheng for very useful conversations and helpful suggestions.



\section{Centres of rigid groups}
\indent In this section, we start by recalling some background on rigid analytic group varieties. Then we will define and construct the centre of a rigid group.

Let $\overline{K}$ be an algebraic closure of $K$.

\subsection{Background on rigid groups}
\begin{defi}
    By a \emph{rigid analytic group variety}, or just a \emph{rigid group}, we mean a group object $G$ in the category of adic spaces locally of topologically finite type over $\Spa K$.
\end{defi}

We refer to \cite{Far19} for some background on rigid groups. We write the group operation multiplicatively as $m\colon G\times G\rightarrow G$, and write $1\in G$ for the identity. Since $K$ is of characteristic $0$, we have the following.
\begin{lemm}[{\cite[proposition 1]{Far19}}]\label{lemm: smoothnessOfRigidGroup}
    Any rigid group $G$ is smooth. Moreover, there is a rigid open subspace $1\in U\subseteq G$ for which there is an isomorphism $U\xrightarrow{\sim}\mathbb{B}^d$ of rigid spaces.
\end{lemm}

In particular, any rigid group $G$ over $K$ is reduced.

\begin{rema}
    On locally strongly Noetherian adic spaces, there is a robust theory of coherent $\cO_X$-modules and Zariski-closed immersions. See \cite[Appendix B]{Zav24} for a detailed treatment.
\end{rema}
\begin{rema}
    Any rigid group $G$ over $K$ is separated, as the identity map $e\colon \Spa K \to G$ is a Zariski-closed immersion. See \cite[remarque 1]{Far19}.
\end{rema}

\begin{prop}\label{prop: ZariskiClosedisClosed}
    Let $i\colon Z \to X$ be a Zariski-closed immersion of locally strongly Noetherian analytic adic spaces. Then $i^{\diamondsuit}\colon Z^{\diamondsuit} \to X^{\diamondsuit}$ is a closed immersion of diamonds.
\end{prop}
\begin{proof}
    By \cite[Proposition 10.11 (i)]{Sch22}, the condition of closed immersions can be checked $v$-locally on the target. We may assume $X$ is affinoid. By \cite[Lemma 5.2]{Zav25_foundation}, for any affinoid perfectoid space $Y$, the map $Z^{\diamondsuit} \times_{X^{\diamondsuit}} Y^{\diamondsuit} \to Y^{\diamondsuit}$ is represented by a Zariski-closed subset in the sense of \cite[Definition 5.7]{Sch22}. The assertion follows.
\end{proof}

\subsection{Definition and construction of centres}
\begin{defi}
    For any rigid space $X$ over $K$, we put
    \[
        X(\overline{K}) \coloneqq \bigcup_{L \subset \overline{K}\text{, $L/K$ finite}} X(L).
    \]
\end{defi}

\begin{defi}
    Let $G$ be a rigid group over $K$. The \emph{centre} of $G$ is a rigid Zariski-closed subgroup $Z$ of $G$ such that for any finite extension $L/K$, the set $Z(L)$ of $L$-points of $Z$ consists of $g\in G(L)$ which acts trivially on $G_L\coloneqq G\times_{\Spa(K)}\Spa(L)$ by conjugation. 
\end{defi}
\begin{rema}
    The centre of $G$ is unique if it exists. Suppose $Z$ and $Z'$ are centres of $G$. Consider projections $Z \times_G Z' \to Z$ and $Z \times_G Z' \to Z'$. These two maps are isomorphisms on $L$-points for any finite extension $L$ of $K$, hence are isomorphisms since $Z,Z'$ and $Z \times_G Z'$ are rigid groups, and in particular separated and reduced.
\end{rema}

\begin{prop}\label{prop: centreIsRigidZariskiClosedSubgroup}
    Let $G$ be a rigid group over $K$. Then the centre $Z$ of $G$ exists.
\end{prop}
\begin{proof}
    Let $L/K$ be a finite Galois extension inside $\overline{K}$. For any $g \in G(L)$, we let
    \[
        \Inn_g\colon G_L \to G_L
    \]
    be the inner action on $G_L$ given by $g$, where $G_L \coloneqq G \times_{\Spa K} \Spa L$. We let
    \[
        N_L \coloneqq \bigcap_{g \in G(L)} \Ker \Inn_g.
    \]
    Then $N_L$ is a Zariski-closed subgroup of $G_L$ as each $\Ker \Inn_g$ is so. Further, $N_L$ is invariant under Galois action. By \'etale descent of coherent ideal sheaves (\cite[Corollary 3.2.3]{deJ96}), $N_L$ descends to a Zariski-closed subgroup $Z_L$ of $G$. By definition, $Z_L(\overline{K})$ consists of elements of $G(\overline{K})$ which commute with $L$-points $G(L)\subset G(\overline{K})$ of $G$.

    Suppose we have a tower of finite Galois extensions $K \subset L \subset E$ inside $\overline{K}$. By definition, there is a Zariski-closed immersion
    \[
        N_E \to N_L \times_{\Spa L} \Spa E,
    \]
    which descends to a Zariski-closed immersion $Z_E \hookrightarrow Z_L$. This shows that our construction of $Z_{\bullet}$ is functorial with respect to the finite Galois extensions. Let
    \[
        Z \coloneqq \bigcap_{\text{$L/K$ finite Galois, $L \subset \overline{K}$}.} Z_L.
    \]
    We show that $Z$ is the centre of $G$. Let $c\colon Z \times_K G \to G, (s,g) \mapsto sgs^{-1}$ be the conjugation action of $Z$ on $G$. For any finite Galois extension $L/K$, the map $c(L)$ coincides with the projection to $G(L)$. Since $Z\times_K G$ and $G$ are reduced separated rigid spaces over $K$, this shows that $c\colon Z \times_K G \to G$ coincides with the projection to $G$. Thus, for any finite extension $L/K$, and any $g \in Z(L)$, $g$ acts trivially on $G_L$. Conversely, let $g$ be an element of $G(L)$ such that $g$ acts trivially on $G_L$ by conjugation. Then $g$ commutes with elements of $G(E)$ for all finite Galois extension $E$ of $K$ containing $L$. By construction, $g \in Z(L)$.
\end{proof}
\begin{rema}
    For an algebraically closed field $K$ and a connected rigid group $G$, this is due to Heuer--Werner--Zhang (\cite[Lemma 2.9]{Heuer_Werner_Zhang_2025}), where the centre of $G$ is defined to be the kernel of the adjoint morphism (\cite[\S 3.1]{Heu22})
    \[
        \mathop{\mathrm{adj}}\colon G \to \Aut(\Lie(G)).
    \]
\end{rema}

\begin{coro}\label{coro: algebraicallyClosedPointOfCentre}
    For any complete algebraically closed extension $K'$ of $K$, and any open bounded valuation subring $K'^+\subset K'$ containing $K^\circ$, the set $Z(K',K'^+)$ is the centre of the abstract group $G(K',K'^+)$.
\end{coro}
\begin{proof}
    The map $c\colon Z \times_K G \to G, (s,g) \mapsto sgs^{-1}$ coincides with the projection to $G$ as shown in the proof of Proposition \ref{prop: centreIsRigidZariskiClosedSubgroup}. Thus, for any $(K',K'^+)$, the subgroup $Z(K',K'^+)$ is contained in the centre of the abstract group $G(K', K'^+)$. Conversely, for any finite Galois extension $L$ of $K$, we have $L^\circ \subset K'^+$ since $L^\circ$ is integral over $K^{\circ}$.  There is then an inclusion $G(L) \subset G(K', K'^+)$. By construction, $Z(K', K'^+)$ consists of elements of $G(K', K'^+)$ which commute with elements of $G(L)$ for all $L/K$ finite Galois. This shows that $Z(K', K'^+)$ contains the centre of the abstract group $G(K', K'^+)$.
\end{proof}

\subsection{Compatibility with \'etale topology and $v$-topology}\label{ssec: centreCompatibilityEtaleAndV}
Let $X$ be an adic space over $K$, viewed as a diamond on $\Perfd_K$. We let $X_v$ (\resp $X_\et$) be the $v$-site (\resp \'etale site) of $X$. There is a morphism of sites $\nu\colon X_v \to X_\et$.

Let $G$ be a rigid group over $K$. We view it as a diamond $G^\diamondsuit$ on $\Perfd_K$. By restriction, it defines a group sheaf $G_v$ on $X_v$ and a group sheaf $G_\et$ on $X_\et$. Clearly $G_\et = \nu_* G_v$.

Recall that for a group object $H$ in a topos $T$, the \emph{centre $Z(H)$} of $H$ is defined to be the kernel of the adjoint action $\mathop{\mathrm{adj}}\colon H \to \SAut(H)$.

Let $Z$ be the centre of the rigid group $G$. Then $Z$ is a Zariski-closed rigid subgroup of $G$ by Proposition \ref{prop: centreIsRigidZariskiClosedSubgroup}. We aim to show that the group sheaf $Z_v$ (\resp $Z_\et$) defined by $Z$ is the centre of the group sheaf $G_v$ (\resp $G_\et$).

This is easy for $v$-topology.
\begin{prop}\label{prop: centreDiamonds}
    The group sheaf $Z^\diamondsuit$ is the centre of $G^\diamondsuit$.
\end{prop}
\begin{proof}
    The conjugation action of $Z$ on $G$ is trivial as shown in the proof of Proposition \ref{prop: centreIsRigidZariskiClosedSubgroup}. Thus, we have an inclusion $Z^{\diamondsuit} \subset Z(G^{\diamondsuit})$. Consider the diagram
    \[\begin{tikzcd}
        {Z^\diamondsuit} & {Z(G^\diamondsuit)} \\
        & {G^\diamondsuit}
        \arrow["j", hook, from=1-1, to=1-2]
        \arrow["i"', hook, from=1-1, to=2-2]
        \arrow["{i'}", hook, from=1-2, to=2-2]
    \end{tikzcd}\]
    By \cite[Proposition 11.10]{Sch22}, $Z(G^\diamondsuit)$ is a diamond. The injection $i$ is a closed immersion by Proposition \ref{prop: ZariskiClosedisClosed}. Hence, $j$ is also a closed immersion as the diagonal of $i'$ is an isomorphism. We only need to show $|j|$ is surjection. Then $|j|$ is a homeomorphism since it is closed map, and we conclude by \cite[Proposition 12.15 (iii)]{Sch22}. It suffices to show that for any $(C,C^+)$ over $K$, where $C$ is a complete algebraically closed perfectoid field, $C^+$ is a open bounded valuation subring of $C$, any $K$-morphism $\Spa(C,C^+) \to Z(G^{\diamondsuit})$ lifts to $\Spa(C,C^+) \to Z^{\diamondsuit}$.

    Let $g \in G^{\diamondsuit}(C,C^+) = G(C,C^+)$. Suppose $g$ acts trivially on $G^{\diamondsuit}|_{(C,C^+)}$. Then $g$ belongs to the centre of the abstract group $G(C,C^+)$. By Corollary \ref{coro: algebraicallyClosedPointOfCentre}, we then have $g \in Z^{\diamondsuit}(C,C^+)$. This shows that $Z^{\diamondsuit}(C,C^+) \simeq Z(G^{\diamondsuit})(C,C^+)$.
\end{proof}

\begin{coro}\label{coro: vCentreOfGv}
    The group sheaf $Z_v$ is the centre of $G_v$.
\end{coro}
This follows by restriction to $X_v$.

To show that $Z_\et$ is the centre of $G_\et$, we need the following lemma, which is an analogue of \cite[(IV{;} 11.10.6)]{EGA}.

\begin{lemm}\label{lemm: analyticallyDenseFamily}
    Let $S$ and $H$ be smooth rigid spaces over $K$. Assume $H$ is separated. Let \{$\Spa(K_i) \to S$\} be the family of all classical rigid points of $S$. Let $T$ be an adic space over $K$. For any $i$, consider the base change
    \[\begin{tikzcd}
        {\Spa(K_i)} & {T_i} \\
        S & {S_T} \\
        {\Spa(K)} & T
        \arrow[from=1-1, to=2-1]
        \arrow[from=1-2, to=1-1]
        \arrow["\lrcorner"{anchor=center, pos=0.125, rotate=-90}, draw=none, from=1-2, to=2-1]
        \arrow["{\phi_i}", from=1-2, to=2-2]
        \arrow[from=2-1, to=3-1]
        \arrow[from=2-2, to=2-1]
        \arrow["\lrcorner"{anchor=center, pos=0.125, rotate=-90}, draw=none, from=2-2, to=3-1]
        \arrow[from=2-2, to=3-2]
        \arrow[from=3-2, to=3-1]
    \end{tikzcd}\]
    Then for any two maps $f^{\diamondsuit},g^{\diamondsuit}\colon S_T^{\diamondsuit} \rightrightarrows H_T^{\diamondsuit}$ over $T^{\diamondsuit}$, if $f^{\diamondsuit}\phi_i^{\diamondsuit} = g^{\diamondsuit}\phi_i^{\diamondsuit}$ for all $i$, then $f^{\diamondsuit} = g^{\diamondsuit}$.
\end{lemm}
Note that all those $T_i$ are finite \'etale over $T$.
\begin{proof}
    Let $i^{\diamondsuit}\colon \Ker(f^{\diamondsuit},g^{\diamondsuit}) \to S_T^{\diamondsuit}$ be the equalizer of $f^{\diamondsuit}$ and $g^{\diamondsuit}$. We only need to show $i^{\diamondsuit}$ is an isomorphism. By Proposition \ref{prop: ZariskiClosedisClosed}, $H^{\diamondsuit} \to H^{\diamondsuit}\times_{K^{\diamondsuit}} H^{\diamondsuit}$ is a closed immersion of diamonds. We see that $i^{\diamondsuit}$ is also. It remains to show that $|i^{\diamondsuit}|$ is surjective, as it would imply that $|i^{\diamondsuit}|$ is an homeomorphism, and one then deduces that $i^{\diamondsuit}$ is an isomorphism by \cite[Proposition 12.15 (iii)]{Sch22}. It suffices to show that for any $p\colon \Spa(C,C^+) \to T$, $i^{\diamondsuit} \times_{T^{\diamondsuit}} \Spa(C,C^{+})^{\diamondsuit}$ is an isomorphism, where $C$ is a complete algebraically closed perfectoid field, and $C^+$ is a open bounded valuation subring of $C$. We then assume that $T = \Spa(C,C^{+})$.

    As $S$ and $H$ are smooth, we have
    \[
        \Hom_{T}(S_T,H_T) \xrightarrow{\sim} \Hom_{T^{\diamondsuit}}(S_T^{\diamondsuit}, H_T^{\diamondsuit})
    \]
    by \cite[Proposition 10.2.3]{Berkely}. Thus, $f^{\diamondsuit}$ and $g^{\diamondsuit}$ come from morphisms of adic spaces $f,g\colon S_T \to H_T$, and we have $f\phi_i = g\phi_i$ for all $i$. We will show that $f=g$ in this case.

    The question is local on $S$, we may assume $S = \Spa(A,A^+)$ is an affinoid adic space. The adic spaces $S$, $H$, $S_T$ and $H_T$ are locally strongly Noetherian. By \cite[Lemma B.6.7, Corollary B.6.9]{Zav24}, $H_T$ is separated over $T$. Thus, $i\colon \Ker(f,g) \to S_T$ is a Zariski-closed immersion, i.e. $i$ is represented as $\Spa(A\hotimes_K C/I, (A\hotimes_K C/I)^+) \to \Spa(A\hotimes_K C, (A\hotimes_K C)^+)$ for an ideal $I$ of $A \hotimes_K C$. We only need to show $A \hotimes_K C \to \prod_{x \in \Spm A} (\kappa_x \hotimes_K C)$ is an injection, where $\kappa_x \coloneqq A/\fm_x$.

    Classical rigid points of $S$ corresponds to maximal ideals of $A$. Since $A$ is Jacobson and reduced, the map
    \[
        A \to \prod_{x \in \Spm A} \kappa_x
    \]
    is injective. The map $A \hotimes_K C \to \prod_{x \in \Spm A} (\kappa_x \hotimes_K C)$ factors as a composition
    \[
        A\hotimes_K C \to \left(\prod_{x \in \Spm A} \kappa_x\right) \hotimes_K C \to \prod_{x \in \Spm A} (\kappa_x \hotimes_K C),
    \]
    where the first arrow is an injection since $-\hotimes_K C$ is exact \cite[\S 2, no. 2, théorème 1]{Gru66}. It remains to show the second arrow is an injection.

    We can write $C$ as a countably filtered colimit of countably generated $K$-Banach algebras. Since countably filtered colimits commute with completion, we may assume $C$ is a countably generated $K$-Banach algebra. By \cite[§2.7, Theorem 4]{BGR84}, $C$ is topologically free. The assertion follows.
\end{proof}

\begin{prop}\label{prop: etCentreOfGet}
    The group sheaf $Z_\et$ is the centre of $G_\et$.
\end{prop}
\begin{proof}
    Take $T=X$, $S=G$ in Lemma \ref{lemm: analyticallyDenseFamily}, let $X_i \to G_X$ be as in the lemma. Clearly we have $Z_\et \subset Z(G_\et)$ since the conjugation action $Z \times_K G \to G$ coincides with the projection to $G$. Conversely, for any $Y \in X_\et$, $g \in Z(G_\et)(Y) \subset G_{\et}(Y)$, $g$ acts trivially on $G_\et|_Y$ by conjugation. We only need to show $\Inn_g$ acts trivially on $G_Y^{\diamondsuit}$, then we conclude by Proposition \ref{prop: centreDiamonds}. Without loss of generality, we assume $Y=X$. For each $i$, the two compositions
    \[\begin{tikzcd}
        {X_i^{\diamondsuit}} & {G_X^{\diamondsuit}} & {G_X^{\diamondsuit}}
        \arrow[from=1-1, to=1-2]
        \arrow["{\Inn_g}", shift left, from=1-2, to=1-3]
        \arrow["{\id}"', shift right, from=1-2, to=1-3]
    \end{tikzcd}\]
    coincide since $X_i \in X_\et$. We then conclude by Lemma \ref{lemm: analyticallyDenseFamily}.
\end{proof}



\section{Comparison of $G$-gerbes}
We commence with the general cases.

\subsection{$G$-gerbes on sites}

\begin{defi}
    Let $E$ be a site, and let $G$ be a group sheaf on $E$. We denote by $\Gerb(E;G)$ the 2-category consisting of $G$-gerbes \cite[IV, definition 2.2.2]{Gir71} on $E$.
\end{defi}

\begin{rema}
    \begin{enumerate}
        \item A $G$-gerbe on a site $E$ is a stack on groupoids on $E$ that is locally isomorphic to $\bbB G$ with trivial band class in $H^1(E, \operatorname{Out}(G))$.

        \item Let $E$ be a site, and let $G$ be a group on $E$. Let $\cF, \cG$ be $G$-gerbes on $E$. Then the stack $\SHom_G(\cF,\cG)$ on $E$ is a $Z(G)$-gerbe. See \cite[IV,theoreme 2.3.2]{Gir71}.
    \end{enumerate}
\end{rema}

\begin{defi}
    Let $f\colon E' \to E$ be a morphism of sites, and let $G$ be a group sheaf on $E'$. We let $\Gerb^f(E';G)$ be the full subcategory of $\Gerb(E';G)$ consisting of $G$-gerbes on $E'$ which comes from $E$ \cite[V, 3.1.8.4]{Gir71}. Namely, a $G$-gerbe $\cG$ lies in $\Gerb^f(E';G)$ if and only if there is a covering sieve $R$ of $E$ such that $\cG$ is trivial on $f^{-1}(R)$. If there is no risk of confusion, we write $\Gerb^E(E';G) = \Gerb^f(E';G)$.
\end{defi}

The following proposition comes from Giraud \cite[V, corollaire 3.1.10.3]{Gir71}. We restate it here.
\begin{prop}[Giraud]\label{prop: GiraudResultOfCentre}
    Let $f\colon E' \to E$ be a morphism of sites. Let $G$ be a group sheaf on $E'$, and let $Z(G)$ be its centre. Suppose that the morphism
    \[
        H^2(E, f_*Z(G)) \to H^2(E', Z(G))
    \]
    is surjective. Then any $G$-gerbe $\cG$ on $E'$ comes from a gerbe on $E$.
\end{prop}

\begin{coro}
    Under the situation above, we have $\Gerb^E(E';G) \cong \Gerb(E';G)$.
\end{coro}

\subsection{Direct image of outer isomorphisms}
For group sheaf $G$ on site $E$, we let $\SInn(G)$ be the sheaf of inner automorphisms of $G$ on $E$. By definition, $\SInn(G) \cong G/Z(G)$, where $Z(G)$ is the centre of $G$.

\begin{lemm}
    Let $f\colon E' \to E$ be a morphism of sites, and let $G$ be a group sheaf on $E'$. Let $Z(G)$ be its centre. Suppose $R^1f_* Z(G) = 0$. Then there is a natural morphism $f_*\colon f_*\SInn(G) \to \SInn(f_*G)$ making the diagram
    \[\begin{tikzcd}
        {f_*G} & {f_*\SInn(G)} \\
        {f_*G} & {\SInn(f_*G)}
        \arrow[from=1-1, to=1-2]
        \arrow[equals, from=1-1, to=2-1]
        \arrow["{f_*}", from=1-2, to=2-2]
        \arrow[from=2-1, to=2-2]
    \end{tikzcd}\]
    commute.
\end{lemm}
\begin{proof}
    Since $R^1f_*Z(G) = 0$, we have a diagram
    \[\begin{tikzcd}
        0 & {f_*Z(G)} & {f_*G} & {f_*\SInn(G)} & 1 \\
        0 & {Z(f_*G)} & {f_*G} & {\SInn(f_*G)} & 1
        \arrow[from=1-1, to=1-2]
        \arrow[from=1-2, to=1-3]
        \arrow[from=1-2, to=2-2]
        \arrow[from=1-3, to=1-4]
        \arrow[equals, from=1-3, to=2-3]
        \arrow[from=1-4, to=1-5]
        \arrow[from=2-1, to=2-2]
        \arrow[from=2-2, to=2-3]
        \arrow[from=2-3, to=2-4]
        \arrow[from=2-4, to=2-5]
    \end{tikzcd}\]
    with two lines exact. The assertion follows.
\end{proof}

\begin{lemm}\label{lemm: directImageOfOutIso}
    Let $f\colon E' \to E$ be a morphism of sites. Let $G$ be a group sheaf on $E'$, and let $Z(G)$ be its centre. Suppose that $R^1f_*Z(G) = 0$ and $R^1f_* \SInn(G) = 1$. Then there is a natural morphism
    \[
        f_*\colon f_*\SAut_{E'}(\band(G)) \to \SAut_{E}(\band(f_*G)),
    \]
    making the diagram
    \[\begin{tikzcd}
        {f_*\SAut_{E'}(G)} & {f_*\SAut_{E'}(\band(G))} \\
        {\SAut_{E}(f_*G)} & {\SAut_E(\band(f_*G))}
        \arrow[from=1-1, to=1-2]
        \arrow["{f_*}"', from=1-1, to=2-1]
        \arrow["{f_*}", from=1-2, to=2-2]
        \arrow[from=2-1, to=2-2]
    \end{tikzcd}\]
    commute.
\end{lemm}
\begin{proof}
    By \cite[IV, corollaire 1.1.7.3]{Gir71}, the map $\SAut_{E'}(G) \to \SAut_{E'}(\band(G))$ induces an isomorphism $\SAut_{E'}(\band(G)) \cong \SOut_{E'}(G) \cong \SAut_{E'}(G)/\SInn(G)$. Since $\SInn(G)$ acts freely on $\SAut_{E'}(G)$, we have a commutative diagram
    \[\begin{tikzcd}
        1 & {f_*\SInn(G)} & {f_*\SAut_{E'}(G)} & {f_*\SAut_{E'}(\band(G))} & 1 \\
        1 & {\SInn(f_*G)} & {\SAut_{E}(f_*G)} & {\SAut_{E}(\band(f_*G))} & 1
        \arrow[from=1-1, to=1-2]
        \arrow[from=1-2, to=1-3]
        \arrow["{f_*}", from=1-2, to=2-2]
        \arrow[from=1-3, to=1-4]
        \arrow["{f_*}", from=1-3, to=2-3]
        \arrow[from=1-4, to=1-5]
        \arrow[from=2-1, to=2-2]
        \arrow[from=2-2, to=2-3]
        \arrow[from=2-3, to=2-4]
        \arrow[from=2-4, to=2-5]
    \end{tikzcd}\]
    with two lines exact. The assertion follows.
\end{proof}

\begin{rema}\label{rema: conditionsOfCentreIsEnough}
    The conditions of Lemma \ref{lemm: directImageOfOutIso} is satisfied when $R^1f_* Z(G) = 0$, $R^1f_* G = 1$ and $R^2f_* Z(G) = 0$.
\end{rema}

\subsection{Comparison of $\Gerb^E(E';G)$ with $\Gerb(E;f_*G)$}
Let $f\colon E' \to E$ be a morphism of sites. Let $G$ be a group sheaf on $E'$, and let $Z(G)$ be its centre. We construct an equivalence between $\Gerb^E(E';G)$ and $\Gerb(E;f_*G)$ under certain conditions.

\begin{prop}
    Suppose that $R^1f_* G = 1$, $R^1f_*Z(G) = 0$ and $R^1f_* \SInn(G) = 1$. Then for any $\cG \in \Gerb^E(E';G)$, the stack $f_*\cG$ is an $f_*G$-gerbe on $E$.
\end{prop}
\begin{proof}
    The condition $R^1f_* G = 1$ implies that we have an equivalence
    \[
        f_*\bbB G \xrightarrow{\sim} \bbB f_*G.
    \]
    We first show that $f_*\cG$ is a gerbe on $E$. Clearly $f_*\cG$ is locally non-empty by definition. To show $f_*\cG$ is locally connected, we may assume $\cG = \bbB G$ since the question is local on $E$. The assertion then follows from the equivalence $f_*\bbB G \xrightarrow{\sim} \bbB f_*G$.

    We now prove $f_*\cG$ is an $f_*G$-gerbe. By definition, there is an action $a$ of $G$ on $\cG$ such that for any local section $x \in \cG(S')$, $S' \in E'$ the map
    \[
        a(x)\colon \band(G)|_{S'} \to \band(\SAut(x))
    \]
    is an isomorphism of bands (see \cite[IV, définition 2.2.2]{Gir71}). For any $S\in E$ and $x \in \cG(f^{-1}(S))$, since $\SAut(x) \cong G|_{f^{-1}(S)} = f_*G(S)$, by Lemma \ref{lemm: directImageOfOutIso}, we can take direct image to obtain isomorphisms
    \[
        f_*a(x)\colon \band(f_*G)|_S \xrightarrow{\sim} \band(f_*\SAut(x)),
    \]
    which are compatible with restrictions. This gives an $f_*G$-gerbe structure on $f_*\cG$.
\end{proof}

\begin{theo}\label{theo: comparisonOfGerbes}
    Assume that $R^1f_*G = 1$, $R^1f_*Z(G) = 0$, and $R^1f_*\SInn(G) = 1$. Suppose further that the subgroup $f_*Z(G)$ is the centre of $f_*G$. Then the functor $f_*$ induces an equivalence $\Gerb^E(E';G)\xrightarrow{\sim} \Gerb(E;G)$.
\end{theo}
\begin{proof}
    We first show that $f_*$ is fully faithful. For any $\cF, \cG \in \Gerb^E(E';G)$, we need to show that the functor $f_*$ induces an equivalence
    \[
        f_*\SHom_{G}(\cF,\cG) \to \SHom_{f_*G}(f_*\cF, f_*\cG)
    \]
    of stacks. The question is local on $E$, we may assume $\cF \cong \bbB G$ and $\cG \cong \bbB G$ by the definition of $\Gerb^E(E';G)$. By \cite[IV, théorème 2.3.2, (iii)]{Gir71}, $\SHom_{G}(\cF,\cG)$ (\resp $\SHom_{f_*G}(f_*\cF, f_*\cG)$) is a $Z(G)$-gerbe (\resp a $f_*Z(G)$-gerbe). We have isomorphisms (\cite[III, Corollaire 2.2.6]{Gir71})
    \[
        \SHom_{G}(\bbB G,\bbB G) \xrightarrow[\sim]{i} \bbB Z(G), \quad f \mapsto \SIsom(f,\mathrm{id}),
    \]
    and
    \[
        \SHom_{G}(\bbB f_*G,\bbB f_*G) \xrightarrow[\sim]{j} \bbB f_*Z(G), \quad f \mapsto \SIsom(f,\mathrm{id}).
    \]

    We will construct a $2$-morphism in the following diagram.
    \[\begin{tikzcd}
        {f_*\SHom_{G}(\bbB G,\bbB G)} & {f_*\bbB Z(G)} \\
        {\SHom_{f_*G}(\bbB f_*G,\bbB f_*G)} & {\bbB f_*Z(G)}
        \arrow["f_*i", from=1-1, to=1-2]
        \arrow["{f_*}", from=1-1, to=2-1]
        \arrow["a"', between={0.4}{0.6}, Rightarrow, from=1-2, to=2-1]
        \arrow["{f_*}", from=1-2, to=2-2]
        \arrow["j", from=2-1, to=2-2]
    \end{tikzcd}\]
    Then $a$ is an isomorphism since $\bbB f_*Z(G)$ is a groupoid. Since $f_*i$, $j$ are equivalences, and $f_*\colon f_*\bbB Z(G) \to \bbB f_*Z(G)$ is an equivalence by vanishing of $R^1f_* Z(G)$. We deduce that the left vertical map is also an equivalence.

    To provide such a $2$-morphism. For any local morphism $g\colon \bbB G \to \bbB G$, we let
    \[
        a(g)\colon f_*\SIsom(g, \mathrm{id}) \to \SIsom(f_*g,\mathrm{id})
    \]
    be the map of $f_*Z(G)$-torsors induced by $f_*$. This proves that $f_*\colon \Gerb^E(E';G) \to \Gerb(E;G)$ is fully faithful.

    To show $f_*$ is an equivalence, we let $\SGerb^E(E';G)$ (\resp $\SGerb(E;G)$) be the $\infty$-sheaf on $E$ such that for any $S \in E$, $\SGerb^E(E'; G)(S) = \Gerb^{E_{/S}}(E'_{/f^{-1}(S)}; G)$ (\resp $\SGerb(E;G)(S) = \Gerb(E_{/S};G)$). We have a functor
    \[
        f_*\colon \SGerb^E(E';G) \to \SGerb(E;G)
    \]
    that is fully faithful and \emph{locally essentially surjective}, namely, for any $\cF \in \SGerb(E;G)(S)$, $\cF$ is locally isomorphic to $\bbB f_*G = f_*\bbB G$. The theorem then follows by descent.
\end{proof}



\subsection{$G$-gerbes on perfectoid spaces}
Let $K$ be a complete non-archimedean field over $\bbQ_p$, $G$ be a rigid group over $K$. Let $X$ be a perfectoid space over $K$. There is a morphism of sites $\nu\colon X_v \to X_\et$. We let $G_v$, $G_\et$ be as in \ref{ssec: centreCompatibilityEtaleAndV}. Heuer proved in \cite{Heu22} that $R^1\nu_* G_v = 1$, and $R\nu^i_* G_v = 0$ for $i \ge 1$ if $G$ is further a commutative rigid group.

\begin{theo}
    \begin{enumerate}[(1)]
        \item Any $G_v$-gerbe on $X_v$ is \'etale locally trivial.
        \item The functor $\nu_*$ induces an equivalence $\Gerb(X_v; G_v) \xrightarrow{\sim} \Gerb(X_\et,G_\et)$.
    \end{enumerate}
\end{theo}
\begin{proof}
    Let $Z$ be the centre of the rigid group $G$. By Corollary \ref{coro: vCentreOfGv} and Proposition \ref{prop: etCentreOfGet}, $Z_v$ (\resp $Z_\et$) is the centre of group sheaf $G_v$ (\resp $G_\et$). By \cite[Theorem 1.1]{Heu22}, we have $R\nu_* Z_v = Z_\et$ and $R^1\nu_* G_v = 1$. The assertion (1) then follows from Giraud's result (Proposition \ref{prop: GiraudResultOfCentre}), and (2) follows from Remark \ref{rema: conditionsOfCentreIsEnough} and Theorem \ref{theo: comparisonOfGerbes}.
\end{proof}

\begin{rema}\label{rema: perfectoidAssumtionNecessary}
    The assumption that $X$ is a perfectoid space is necessary for our theorems.
    \begin{enumerate}[(1)]
        \item If $X$ is a smooth rigid space  over $K$, with $\Char K = 0$. Heuer shows in \cite[Theorem 1.8]{Heuer_LineBundles} that $\nu_* \bbG_{m,v} = \bbG_{m,\et}$ and that
            \[
                R^i\nu_* \bbG_{m,v} = \Omega^i_{X_{\et}}\{-i\}, \quad \text{for $i \ge 1$}.
            \]
        \item More generally, \cite[Theorem 1.1]{Gerth2024} describes $H^2_v(X,G)$ for certain $X$ and locally $p$-divisible commutative rigid groups $G$.
    \end{enumerate}
    Therefore, $H^2_v(X,G)$ is quite different from $H^2_{\et}(X,G)$ for general $X$ even in the abelian cases.
\end{rema}


\bibliographystyle{amsalpha}
\bibliography{refs}

\end{document}